\documentclass[12pt, reqno]{amsart}
\usepackage{amsmath, amsthm, amscd, amsfonts, amssymb, graphicx, color}
\usepackage[bookmarksnumbered, colorlinks, plainpages]{hyperref}
\hypersetup{colorlinks=true,linkcolor=red, anchorcolor=green, citecolor=cyan, urlcolor=red, filecolor=magenta, pdftoolbar=true}
\newtheorem{theorem}{Theorem}[section]
\newtheorem{lemma}[theorem]{Lemma}
\newtheorem{proposition}[theorem]{Proposition}
\newtheorem{corollary}[theorem]{Corollary}
\theoremstyle{definition}
\newtheorem{definition}[theorem]{Definition}
\newtheorem{example}[theorem]{Example}

\theoremstyle{remark}
\newtheorem{remark}[theorem]{Remark}
\numberwithin{equation}{section}
\newcommand\xxrightarrow[2]{\overset{#1}{\underset{#2}\longrightarrow}}

\begin{document}

\setcounter{page}{1}

\title[ A generalization of order convergence  ]{ A generalization of order convergence}
\author[K. Haghnejad Azar]{Kazem Haghnejad Azar}

\address{Department  of  Mathematics  and  Applications, Faculty  of  Mathematical  Sciences, University of Mohaghegh Ardabili, Ardabil, Iran.}

\email{\textcolor[rgb]{0.00,0.00,0.84}{haghnejad@uma.ac.ir}}


\keywords{ order convergence, $F$-order convergent, $b$-order continuous operator.
\newline \indent $^{*}$Corresponding author}

\begin{abstract} 
Let $E$ be a sublattice of a vector lattice $F$.
$\left( x_\alpha \right)\subseteq E$ is said to be $ F $-order convergent to a vector $ x $  (in symbols $ x_\alpha \xrightarrow{Fo} x $), whenever there exists another net $ \left(y_\alpha\right) $  in $F $ with the some index set satisfying 
$ y_\alpha\downarrow 0 $ in $F$ and $ \vert x_\alpha - x \vert \leq y_\alpha $ for all indexes $ \alpha $.
If $F=E^{\sim\sim}$, this convergence is called $b$-order convergence and we write $  x_\alpha \xrightarrow{bo} x$.
 In this manuscript, first we study some properties of $Fo$-convergence nets and we extend some results  to the general case.  In the second part, we introduce $b$-order continuous operators and we invistegate some properties of this new concept. An operator $T$ between two vector lattices $E$ and $F$ is said to be $b$-order continuous, if $  x_\alpha \xrightarrow{bo} 0  $ in $E$ implies $ Tx_\alpha \xrightarrow{bo} 0$ in $F$. 
\end{abstract} \maketitle

\section{Introduction and preliminaries}
To state our result, we need to fix some notation and recall some definitions.
Let us say that a vector subspace $G$ of an ordered vector space $E$ is majorizing $E$ whenever for each $x\in E$ there exists some $y\in G$ with $x\leq y$. A vector sublattice $G$ of vector lattice $E$ is said to be order dense in $E$ whenever for each $0<x\in E$ there exists some $y\in G$ with $0<y\leq x$. A Dedekind complete vector lattice $E$ is said to be a Dedekind completion of the vector lattice $G$ whenever $E$ is lattice isomorphism to a majorizing order dense sublattice of $E$. A subset $A$ of a vector lattice $E$ is said to be order closed if it follows from $\{x_\alpha\}\subseteq A$ and $x_\alpha\xrightarrow{o} x$ in $E$ that $x\in A$. A vector sublattice $G$ of vector lattice $E$ is said to be regular, if the embedding of $E$ into $F$ preserves arbitrary suprema and infima.
Let $E$, $F$ be vector lattices. An operator $T:E\rightarrow F$ is said to be order bounded if it maps each order bounded subset of $E$ into order bounded subset of $F$. The collection of all order bounded operators from a vector lattice $E$ into a vector lattice $F$ will be denoted by $\mathcal{L}_b(E,F)$. 
The vector space $E^\sim$ of all order bounded linear functionals on vector lattice      $E$ is called the {order dual} of $E$, i.e.,  $E^\sim =L_b(E ,R)$. Let  
$A$
be a  subset  of vector lattice        
$E$ and 
 $Q_E$ be the natural mapping from $E$ into  $E^{\sim \sim}$.
If $Q_E(A)$
is  order  bounded  in 
$E^{\sim \sim}$, 
then  
$A$
is  said  to $b-$order  bounded  
$E$, see \cite{alpay2003property}. 
It  is  clear  that  every  order  bounded  subset  of
$E$
is  
$b-$order bounded. However, the  converse is    not  true  in general. For example, 
$A= \{e_n \mid  n \in N\}$
$b-$order bounded  in 
$c_0$
but  
$A$
is  not  order  bounded  in 
$c_0$.
A linear operator between two  vector lattices is order continuous (resp. $\sigma$-order continuous) if it maps order null nets (resp. sequences) to order null nets (resp. sequences). The collection of all order continuous (resp. $\sigma$-order continuous) linear operators from vector lattice $E$ into vector lattice $F$ will be denoted by $\mathcal{L}_n(E,F)$ (resp. $\mathcal{L}_c(E,F)$). 
For unexplained terminology and facts on  Banach lattices and positive operators, we refer the reader to \cite{1,aliprantis2006positive}.
%%--------------------------------------------------------------------------------%%
%%--------------------------------------------------------------------------------%%
\section{$F$-order convergent on vector lattices}
In all parts of this section $E$ is a vector sublattice of vector lattice $F$.
Let $A\subseteq E$. We say that $\inf A$ exists in $E$ with respect to $F$, if $\inf A$ exits in $F$ and $\inf A\in E$, in this case  we write $\inf_F A$ exists,.
 For a net $(x_\alpha )_\alpha\subseteq E$ and $x\in E$, the notation $ x_\alpha\downarrow_{F} x $ means that $ x_\alpha\downarrow $ and $ \inf \left(x_\alpha\right) = x $ holds in $F$. The meanings of $ x_\alpha\uparrow  $ and $ x_\alpha\uparrow_{F} x $ are analogous. Obviously if $ x_\alpha\downarrow_{F} 0 $, then $ x_\alpha\downarrow 0 $, but as following example the converse in general not holds.
\begin{example}
Assume that $F$ is a set of real valued functions on $[0,1]$ of  form $f=g+h$ where $g$ is continuous and $h$ vanishes except at finitely many point. Let $E=C([0,1])$          and $f_n(t)=t^n$  where $t\in [0,1]$. It is clear that $f_n\downarrow 0$ in $E$ and $\inf_F f_n$  not exists in $E$, but  we have  $f_n\downarrow\chi_{\{1\}}$ in $F$. 
\end{example}

 It is obvious that if $E$ is regular in $F$, then for each net $(x_\alpha )_\alpha\subseteq E$ and $x\in E$, $ x_\alpha\downarrow_{F} x $ if and only if $ x_\alpha\downarrow x $. 

The notation $ x_\alpha\downarrow_{b} x $ means that $ x_\alpha\downarrow $ and $ \inf \left(x_\alpha\right) = x $ holds in  $  E^{\sim\sim} $. The meanings of $ x_\alpha\uparrow_{b} x $ is analogous.

\begin{definition}
 $\left( x_\alpha \right)\subseteq E$ is said to be $ F $-order convergent (in short {\bf $Fo$-convergent}) to a vector $ x\in E$  (in symbols $ x_\alpha \xrightarrow{Fo} x $ ), whenever there exists another net $ \left(y_\alpha\right) $  in $F $ with the some index set satisfying $ y_\alpha\downarrow 0 $ and $ \vert x_\alpha - x \vert \leq y_\alpha $ for all indexes $ \alpha $.
\end{definition}

%%%%%%%%%%%%%%%%%%%%%%%%%%%%%%%%%%%%%%%%%%%%%%%%%%%%%%%%%%%%%%%%%%%%%%%%%%%%%%%%%%%%%%%%%%%%%%%%%%%%
In the same way, a net $\left( x_\alpha \right)$ of   $E$ is said to be $ b $-order convergent (in short {\bf  $bo$-convergent}) to a vector $ x $  (in symbols $ x_\alpha \xrightarrow{bo} x $ ), whenever there exists another net $ \left(y_\alpha\right) $  in $ E^{\sim\sim} $ with the some index set satisfying $ y_\alpha\downarrow 0 $ and $ \vert x_\alpha - x \vert \leq y_\alpha $ for all indexes $ \alpha $.

It is clear that every order convergence net in vector lattice $E$ is $F$-order convergent, but as following example the converse in general not holds.
\begin{example}\label{E0}
Suppose that $E=c_{0}$ and $\left(e_{n}\right)$ is the standard basis of $c_{0}$. We know that $\left(e_{n}\right)$ is not order convergence to zero, but $\left(e_{n}\right)$ is $\ell^\infty$-order (or $b$-order) convergent to zero.
\end{example}

 It can easily be seen that a net in vector lattice $E$ can have at most one $F $-order limit. The basic properties of $ Fo $-convergent are summarized in the next theorem.

\begin{theorem}
Assume that the nets $ \left(x_\alpha\right) $ and  $ \left(y_{\beta}\right) $ of a vector lattice $E$ satisfy $ x_\alpha \xrightarrow{Fo} x $ and $ y_{\beta} \xrightarrow{Fo} y $. Then we have 
\begin{enumerate}
\item $ \vert x_\alpha \vert \xrightarrow{Fo} \vert x \vert $; $ x^{+}_\alpha \xrightarrow{Fo} x^{+} $ and $ x^{-}_\alpha \xrightarrow{Fo} x^{-} $.
\item $ \lambda x_\alpha + \mu y_{\beta} \xxrightarrow{Fo}{(\alpha,\beta)} \lambda x + \mu y $ for all  $ \lambda , \mu \in \mathbb{R} $.
\item $ x_\alpha \vee y_{\beta} \xxrightarrow{Fo}{(\alpha,\beta)} x \vee y $ and $ x_\alpha \wedge y_{\beta} \xxrightarrow{Fo}{(\alpha,\beta)} x \wedge y $.
\item For each $z\in F$, if $ x_\alpha \leq z $ for all $ \alpha \geq \alpha_{1} $, then $ x \leq z $. 
\item If  $ 0 \leq x_\alpha \leq  y_\alpha $, then $ 0\leq x \leq y $.
\item If $ P $ is order projection, then $ Px_\alpha \xrightarrow{Fo} Px $. 
\end{enumerate} 
\end{theorem}

%%%%%%%%%%%%%%%%%%%%%%%%%%%%%%%%%%%%%%%%%%%%%%%%%%%%%%%%%%%%%%%%%%%%%%%%%%%%%%%%%%%%%

\begin{definition}
 $ A \subseteq E$ is said to be $F$-order closed whenever $ \left(x_\alpha\right) \subseteq A $ and $ x_\alpha \xrightarrow{Fo} x $ imply $ x\in A $. 
\end{definition}
The set $A\subseteq E$ is $F$-order closed means that $A$ is order closed with respect to vector lattice $F$. If $A\subseteq E$ is order closed in $E$, then it is clear that $A$ is $F$-order closed, but but the converse in general not holds. 
 For example $c_{0}$ is  order closed, but is not $\ell^\infty$-order closed.

\begin{lemma}\label{l:2.5}
Let $ A\subseteq E $  be a solid subset of $ F $. Then $ A $ is $F$-order closed if and only if $ \left(x_\alpha\right)\subseteq A $ and $ 0\leq x_\alpha \uparrow_{F} x $ imply $ x \in A $.
\end{lemma}

\begin{proof}
Suppose that $ A $ is a $F$-order closed, and  $ \left(x_\alpha\right)\subseteq A $ and $ 0\leq x_\alpha \uparrow_{F} x $. Therefore  $ 0\leq \vert x_\alpha - x \vert \leq x - x_\alpha \downarrow_{F} 0 $. 
It follows that $ x_\alpha \xrightarrow{Fo} x $, and since $A$ is order closed, implies that $ x \in A $.

Conversely assume that $ \left(x_\alpha\right) \subseteq A $ and $ x_\alpha \xrightarrow{Fo} x $.  Set a net $ \left(y_\alpha\right) $ in $ F $ with same index net satisfying $ y_\alpha\downarrow_{F}0 $ and $ \vert x_\alpha - x \vert \leq y_\alpha $ for each $ \alpha $. Since $ \left( \vert x \vert - y_\alpha \right)^{+} \leq \vert x_\alpha \vert $ for each $ \alpha $ and $A$ is solid, follows that $( \left( \vert x \vert - y_\alpha \right)^{+})_\alpha\subseteq  A$.  Obviously that $ 0\leq \left( \vert x \vert - y_\alpha \right)^{+} \uparrow_{b} \vert x \vert $, and so by hypothesis we have $ x\in A $. Its follows that  $ A $ is $F $-order closed.
\end{proof}
\begin{definition}
\begin{enumerate}
\item $E$ is said to be $F$-Dedekind (or $F$-order) complete, if every nonempty $A\subseteq E$ that is bounded from above in $F$ has supermum in $E$. In case $F=E^{\sim\sim}$,  we say that $E$ is $b$-Dedekind complete.
\item A subset $A$ of $E$ is called $F$-order bounded, if $A$ is order bounded in $F$. In case $F=E^{\sim\sim}$,  we say that $A$ is  $b$-order bounded.
\item If each $F$-order bounded subset of $E$ is order bounded in $E$, then $E$ is said to have the $F$-property. In case $F=E^{\sim\sim}$,  we say that $E$ has $b$-property.
\end{enumerate}
\end{definition}
\begin{remark}
Every majorizing sublattice $E$ of $F$ has the $F$-property. Since $E^\sim$ has $b$-property, $E^\sim$ is $b$-Dedekind complete.
If $E$ is  $F$-Dedekind  complete, then $E$ is Dedekind complete, but the converse in general not holds, of course $c_0$ is Dedekind complete, but is not     $\ell^\infty$-Dedekind complete. Let $K$ be a compact Hausdorff space and let $C(K)$ and  $B(K)$  be vector lattices of real valued continuous and bounded functions on $K$, respectively,  under pointwise order and algebric operations. By easy calculation, it is obvious that $C(K)$ is both $B(K)$-Dedekind complete and $b$-Dedekind complete.
It is clear that $E$ is $F$-Dedekind  complete if and only if $E$ is  Dedekind complete  with $F$-property. It is easy to show that a vector lattice $E$ has $F$-property if and only if for each net $(x_\alpha )$ in $E$ with $x_\alpha\uparrow\leq y$ for some $y\in F$, $(x_\alpha )$ is order bounded in $E$.
\end{remark}

 Let $E$ be a vector sublattice of $F$ and $I$ be an ideal in $E$.  In general,  $I$ is not an ideal in $F$. For example, set $F=\mathbb{R}^3$ and define the order on $F$ in the following way:
$$x=(x_1,x_2,x_3)< y=(y_1,y_2,y_3)$$
whenever one of the following relations hold
\begin{enumerate}
 \item $x_1< y_1$ or; 
 \item $x_1=y_1, ~x_2<y_2$ or; 
 \item $x_1=y_1,~x_2=y_2,~x_3<y_3$.
\end{enumerate}
  It is clear $F$ with this order is a vector lattice. Now if we take $E=\{(x,y,0):~x,y\in \mathbb{R}\}$ and $I=\{(0,y,0):~y\in \mathbb{R}\}$, then obviously that $I$ is an ideal in $E$, but is not ideal in $F$. \\
The above example shows that the property of being ideal thus depend on the space in which $I$ is embedded. Now if $E$ is $F$-Dedekind complete and order dense in $F$,  the following theorem  shows that $I$ is an ideal in $E$ if and only if $I$ is an ideal in $F$.  

\begin{theorem}\label{2.8}
Assume that $E$ is $F$-Dedekind complete. The following statements hold.
\begin{enumerate}
\item  Each $F$-order convergence net in $E$ is order convergent in $E$.
\item If $E$ is order dense in $F$, then   $B$ is a band in $E$ if and only if $B$ is a band in $F$.
\item By assumption (2), if  $F=E^{\sim\sim}$, then $E^\delta=E^{\sim\sim}$.
\item If $F=E^{\sim\sim}$ and $E^\sim=E_n^\sim$, then $E$ is perfect.
\end{enumerate}
\end{theorem}
\begin{proof}
\begin{enumerate}
\item Assume that $(x_\alpha )_\alpha\subseteq E$ is $F$-order convergent to $x$ in $E$. Then there exists a net $(y_\alpha)_\alpha\subseteq F$ such that  $ y_\alpha\downarrow 0 $ and $ \vert x_\alpha - x \vert \leqslant y_\alpha $ for all indexes $ \alpha $. Set $z_\alpha= \vert x_\alpha - x \vert$ and take $w_\alpha =\bigvee_{\beta \geqslant \alpha}z_\beta$. Since $E$ is $F$-Dedekind complete, $(w_\alpha)_\alpha\subseteq E$.  It follows that $ \vert x_\alpha - x \vert \leqslant w_\alpha\leqslant y_\alpha$ and $w_\alpha\downarrow 0$ in $E$. Thus $(x_\alpha )_\alpha$ is order convergence to $x$ in $E$.

\item If $B$ is a band in $F$, it is clear that $B$ is a band in $E$. Now assume that $B$ is a band in $E$.
First we prove that $I$ is an ideal in $F$.  Let $x\in F$ and $y\in I$ such that $0<\vert x\vert<\vert y\vert$. Since $E$ is order dense in $F$, there is a $z\in E$ such that $0<z\leqslant x^+\leqslant \vert x\vert$, which follows that $z\in B$. Put
$\sup\{z\in B:~0<z\leqslant x^+\}=w$. By $F$-Dedekind completeness of $E$, we have $w\in E$, and so $w\in B$. If $w<x^+$, then $0<x^+-w$. Since $E$ is order dense in $F$,   there is $v\in E$ such that $0< v< x^+-w$, and so $0<w+v<x^+$.  It follows that $w+v\in B$, which is impossible. Thus $w=x^+$ belong to $B$. In the same way $x^-\in B$, and so $x\in B$. This shows that $B$ is an ideal in $F$.  Now since $E$ is  $F$-Dedekind complete, by using Lemma \ref{l:2.5},    $B$ is order closed in $F$ and proof follows.
\item By using Proposition 8 from \cite{alpay2009characterizations}, proof follows.
\item by using Corollary 10 from \cite{alpay2009characterizations}, proof follows.
\end{enumerate}
\end{proof}

%%%%%%%%%%%%%%%%%%%%%%%%%%%%%%%%%%%%%%%%%%%%%%%%%%%%%%%%%%%%%%%%%%%%%%%%%%%%%%%%%%%%%%%%%%%%%
%%%%%%%%%%%%%%%%%%%%%%%%%%%%%%%%%%%%%%%%%%%%%%%%%%%%%%%%%%%%%%%%%%%%%%%%%%%%%%%%%%%%%%%%%%%%%

%%%%%%%%%%%%%%%%%%%%%%%%%%%%%%%%%%%%%%%%%%%%%%%%%%%%%%%%%%%%%%%%%%%%%%%%%%%%%%%%%%%%%%%%%%%%%%%%%%%%%%%%%%%%%%%%%%%%%%%%%%%%%%%%%%%%%%%%%%%%%%%%%%%%%%%
%%%%%%%%%%%%%%%%%%%%%%%%%%%%%%%%%%%%%%%%%%%%%%%%%%%%%%%%%%%%%%%%%%%%%%%%%%%%%%%%%%%%%%%%%%%%%%%%%%%%%%%%%%%%%%%%%%%%%%%%%%%%%%%%%%%%%%%%%%%%%%%%%%%%%%%

\section{$b$-order continuous operators}
Let $E$ and $F$ be vector lattices. 
 An operator $T: E\rightarrow F$ is called $b$-order bounded operator if it maps $b$-order bounded subset of $E$ into $b$-order bounded  subset of $F$. 
 The collection of $b$-order bounded operators will be denoted by:
\begin{equation*}
L_{b^\sim}\left(E,F\right):=\lbrace T\in L\left(E,F\right): T \text{\ is $b$-order bounded operator}\rbrace.
\end{equation*}
 An order bounded operator between two vector lattices is $b$-order bounded, but as Example 2.4, \cite{alpay2003property}, the converse, in general, not holds.

\begin{proposition}
Let $E$ and $F$ be vector lattices.  $T\in L_{b^\sim}\left(E,F\right)$ if and only if   $\vert T\vert\in L_{b^\sim}\left(E,F\right)$
\end{proposition}
\begin{proof}
Assume that $T\in L_{b^\sim}\left(E,F\right)$, we shows that
 $ \vert T\vert\in L_{b^\sim}\left(E,F\right)$.
Let $(x_\alpha )_\alpha$ be a net in $E^+$ with $x_\alpha\uparrow x^{\prime\prime}$ for some $x^{\prime\prime}\in E^{\sim\sim}$. Let $A$ be the solid hull of $(x_\alpha )_\alpha$  in $E$. Since $(x_\alpha )_\alpha$ is order bounded in $E^{\sim\sim}$, $A$ is $b$-order bounded in $E$. Then $T(A)$ is $b$-order bounded in $F$. Since $F^{\sim\sim}$ is Dedekind complete, $\sup T(A)$ exists in  $F^{\sim\sim}$. 
Let $y\in E$ with $\vert y\vert\leq x_\alpha$  for fix $\alpha\in I$. Then $y\in A$ and $T(y)\leq \sup T(A)$. It  follows that $\vert T\vert(x_\alpha)\leq\sup T(A)$. By Dedekind completeness of $F^{\sim\sim}$,   $\sup_\alpha \vert T\vert(x_\alpha)$ exists in $F^{\sim\sim}$. This shows that $ \vert T\vert\in L_{b^\sim}\left(E,F\right)$. The converse by easy calculation follows. 
\end{proof}
The above proposition shows that $L_{b^\sim}\left(E,F\right)$ is a lattice, and so is a vector lattice. So it is easy to shows  that  $L_{b^\sim}\left(E,F\right)$ is an ideal in $L\left(E,F\right)$.
\begin{definition} 
An operator $T:E \rightarrow F $ between two  vector lattices is said to be $ b $-order continuous, if $ x_\alpha \xrightarrow{bo} 0 $ in $E$ implies $ Tx_\alpha\xrightarrow{bo} 0 $ in $F$. 
\end{definition}
\begin{proposition}
If $T$ is $ b $-order continuous operator between two  vector lattices $ E $ and $ F $. Then $ T $ is $ b $-order bounded.
\end{proposition}
\begin{proof}
Suppose that $T:E \rightarrow F $ is a $ b $-order continuous operator. Let $ A=\left[ 0,x''\right] \cap E $ for some $ x''\in E^{\sim\sim} $. Let $ \Lambda=\lbrace \beta: 0\leq \beta \leq x'' \rbrace $ and we write $ \alpha \preceq \beta $ if and only if $ \alpha \geq \beta $. We consider a net $ \left( x_\alpha \right)_{\alpha \in \Lambda} $ as follows
\begin{equation*}
x_\alpha =
 \left\{ \begin{array}{ll}
\alpha & \alpha \in A  ;\\
0 &  \alpha \notin A  
\end{array} \right.
\end{equation*}
Therefore $ x_\alpha \xrightarrow{bo} 0 $, since if we set $ y_\alpha=\alpha $ then $ y_\alpha \downarrow_{b} 0$ and $ \vert x_\alpha \vert \leq y_\alpha $. By the $ b $-order continuity of $ T $, there exists a net $ \left(z_\alpha\right) $ of $ F^{\sim\sim}$ with the same index $ \Lambda $ such that $ \vert Tx_\alpha \vert \leq z_\alpha \downarrow_{b} 0$. Consequently, if $ \alpha\in\Lambda $ then we have $ \vert Tx_\alpha \vert \leq z_\alpha \leq z_{x''} $ that $ z_{x''} \in F^{\sim\sim} $, and this show that $ T $ is a $ b $-order bounded operator.
\end{proof}
%%%%%%%%%%%%%%%%%%%%%%%%%%%%%%%%%%%%%%%%%%%%%%%%%%%%%%%%%%%%%%%%%%%%%%%%%%%%%%%%%%%%%%%%%%%%%%%%%%%%%

As above proposition, the class of $b$-order continuous operators is a subspace of 
$ L_{b^\sim}\left(E,F\right)$ and will be denoted by $L_{n^\sim}\left(E,F\right)$, that is 
\begin{equation*}
L_{n^\sim}\left(E,F\right):=\lbrace T\in L_{b^\sim}\left(E,F\right): T \text{\ is $b$-order continuous}\rbrace,
\end{equation*}

%%%%%%%%%%%%%%%%%%%%%%%%%%%%%%%%%%%%%%%%%%%%%%%%%%%%%%%%%%%%%%%%%%%%%%%%%%%%%%%%%%%%%%%%%%%%%%%%%%%%%

%%%%%%%%%%%%%%%%%%%%%%%%%%%%%%%%%%%%%%%%%%%%%%%%%%%%%%%%%%%%%%%%%%%%%%%%%%%%%%%%%%%%%%%%%%%%%%%%%%%%%
\begin{proposition}
Let $E$ and $F$ be vector lattices, $T$ and $S$ are operators from $E$ into $F$ and $0\leq S \leq T $. If $T\in L_{n^\sim}\left(E,F\right)$, then   $S\in L_{n^\sim}\left(E,F\right)$.
\end{proposition}
\begin{proof}
Let $\left(x_\alpha\right) $ be net in $E$ that  $ x_\alpha \xrightarrow{bo} 0 $. Since $T$ is $b$-order continuous, there exists $y_\alpha \in F^{\sim\sim}$ such that $\vert T\vert x_\alpha\vert\vert=T\vert x_\alpha\vert \leq y_\alpha\downarrow_{b} 0 $.
 On the other hand, $\vert S\left(x_\alpha\right)\vert \leq S\vert x_\alpha\vert \leq T\vert x_\alpha\vert \leq y_\alpha \downarrow_{b} 0$ for every $x_\alpha$ in $E$, and this yields that $S$ is $b$-order continuous.  
\end{proof}

\begin{lemma}\label{l:3.5}
 Let $E$ and $F$ be two vector lattices with Dedekin complete. Then 
$0<T\in L_{n^\sim}\left(E,F\right)$ if and only if for each net $\left(x_\alpha\right) $ in $E$, $x_\alpha\downarrow_b 0$ implies 
$Tx_\alpha\downarrow_b 0$.
\end{lemma}
\begin{proof}
Assume that $0<T\in L_{n^\sim}\left(E,F\right)$ and $x_\alpha\downarrow_b 0$. It follows that $ x_\alpha \xrightarrow{bo} 0 $, and so $ Tx_\alpha \xrightarrow{bo} 0 $. Then there is a net $\left(z_\alpha\right) $ in $F^{\sim\sim}$ such that 
 $Tx_\alpha= \vert Tx_\alpha\vert\leq z_\alpha\downarrow 0$, which follows that $Tx_\alpha\downarrow_b 0$.\\
 Conversely, Let $\left(x_\alpha\right)\subseteq E$  such that  $ x_\alpha \xrightarrow{bo} 0 $ in $E$. Then there is a net $\left(y_\alpha\right) $ in $E^{\sim\sim}$ such that 
 $ \vert x_\alpha\vert\leq y_\alpha\downarrow 0$ in $E^{\sim\sim}$. Set $w_\alpha =\bigvee_{\beta\geq \alpha}\vert x_\beta\vert<z_\alpha$. Then we have $w_\alpha\downarrow_b 0$, and so $Tw_\alpha\downarrow_b 0$. Since $\vert Tx_\alpha\vert\leq T \vert  x_\alpha\vert\leq Tw_\alpha$,  $T x_\alpha \xrightarrow{bo} 0 $ in $F$, and proof follows.
\end{proof}

As an application of  Lemma \ref{l:3.5}, we have the following corollary, in which the techniques of this corollary has been similar argument like as Theorem 1.56  \cite{1} and we omit its proof.

\begin{corollary}
Let $E$ and $F$ be two vector lattices with Dedekind complete. Then the following assertions are equivalent.
\begin{enumerate}
\item $ T\in L_{n^\sim}\left(E,F\right)$.
\item $x_\alpha\downarrow_b 0$ implies  $Tx_\alpha\downarrow_b 0$.
\item $x_\alpha\downarrow_b 0$ implies  $\inf_b\vert Tx_\alpha\vert= 0$.
\item $ T^-, ~T^+~\text{and}~\vert T\vert$ belong to $ L_{n^\sim}\left(E,F\right)$.
\end{enumerate}
\end{corollary}

\begin{proposition}
Let $E$ and $F$ be both vector lattices. Then we have the following assertions.
\begin{enumerate}
\item If $F$ is a $b$-Dedekind complete, then $L_b(E,F)=L_{b^\sim}(E,F)$.
\item If $E$ and $F$ are both  $b$-Dedekind complete, then  $ L_{n^\sim}\left(E,F\right)= L_{n}\left(E,F\right)$.
\item   $ L_{n^\sim}\left(E,F\right)$ is a band  of  $L_{b^\sim}\left(E,F\right)$.
\end{enumerate}
\end{proposition}
\begin{proof}
\begin{enumerate}
\item Obviously that $L_b(E,F)\subseteq L_{b^\sim}(E,F)$. Now we prove  $ L_{b^\sim}(E,F)\subseteq L_b(E,F)$. Let $A$ be an order bounded subset of $E$ and $T\in  L_{b^\sim}(E,F)$. Then $T(A)$ is an $b$-order bounded subset of $F$. Since $F$ be $b$-Dedekind complete, follows that $\sup_b T(A)$ exists in $E$. It follows that $T(A)$ is order bounded in $E$, and so proof follows.
\item Assume that $T\in L_{n}\left(E,F\right)$. Let $\left(x_\alpha\right) $ be a net in $E$ that  $ x_\alpha \xrightarrow{bo} 0 $. Since $E$ is $b$-Dedekind complete, by using Theorem \ref{2.8},  $ x_\alpha \xrightarrow{o} 0 $. By assumption, we have  $ Tx_\alpha \xrightarrow{o} 0 $, which follows that $ Tx_\alpha \xrightarrow{bo} 0 $, and so $T\in L_{n^\sim}\left(E,F\right)$.\\
 Now let $T\in L_{n}\left(E,F\right)$ and $\left(x_\alpha\right)\subseteq E$ such that  $ x_\alpha \xrightarrow{o} 0 $. It follows that $ x_\alpha \xrightarrow{bo} 0 $, and so $ Tx_\alpha \xrightarrow{bo} 0 $. By Dedekind completeness of $F$ and another using Theorem  \ref{2.8}, we have $ Tx_\alpha \xrightarrow{bo} 0 $, which follows that     $ L_{n}\left(E,F\right)\subseteq  L_{n^\sim}\left(E,F\right)$, and proof down.
\item Proof has the similar argument from Theorem 1.57 \cite{1}.
\end{enumerate}
\end{proof}

\end{document}